\documentclass[11pt,reqno]{amsart}

\usepackage{setspace}
\usepackage{footnote}
\usepackage{graphicx}
\usepackage{cmap,mathtools}
\usepackage[T1]{fontenc}
\usepackage[utf8]{inputenc}
\usepackage[english]{babel}
\usepackage{amsfonts,amssymb,amsthm,amsmath,enumerate,bm,cite, mathrsfs}
\usepackage{url}
\usepackage[shortlabels]{enumitem}
\usepackage{float}
\usepackage{secdot}

\allowdisplaybreaks
\usepackage{graphicx}
\usepackage{epstopdf}
\usepackage{a4wide}
\usepackage{xparse}
\usepackage[dvipsnames]{xcolor}
\usepackage[colorlinks=true,linktocpage,pdfpagelabels,bookmarksnumbered,bookmarksopen]{hyperref}
\hypersetup{
	colorlinks=true,
	linkcolor=blue,         
	citecolor=red,
	urlcolor=black
}

\usepackage[margin=1in]{geometry}

\usepackage{orcidlink}  
\mathtoolsset{showonlyrefs}

\newtheorem{theorem}{Theorem}
\newtheorem{proposition}[theorem]{Proposition}
\newtheorem{lemma}[theorem]{Lemma}

\theoremstyle{definition}
\newtheorem{definition}[theorem]{Definition}
\newtheorem{remark}[theorem]{Remark}

\let\OLDthebibliography\thebibliography
\renewcommand\thebibliography[1]{
	\OLDthebibliography{#1}
	\setlength{\parskip}{1pt}
	\setlength{\itemsep}{1pt plus 0.3ex}
}

\numberwithin{equation}{section} 
\numberwithin{theorem}{section}

\DeclarePairedDelimiter\norm{\lVert}{\rVert}%
\makeatletter
\let\oldnorm\norm
\def\norm{\@ifstar{\oldnorm}{\oldnorm*}}

\newcommand{\al} {\alpha}

\newcommand{\De} {\Delta}

\newcommand{\om} {\omega}
\newcommand{\Om} {\Omega}

\newcommand{\la} {\lambda}
\newcommand{\La} {\Lambda}

\newcommand{\ep} {\epsilon}
\newcommand{\ra} {\rightarrow}

\newcommand{\intRn}{\displaystyle{\int_{\mathbb{R}^N}}}

\newcommand\restr[2]{{
  \left.\kern-\nulldelimiterspace 
  #1 
  \right|_{#2} 
  }}

\def\A{{\mathcal A}}
\def\C{{\mathcal C}}

\def\Ds{{\mathscr{D}}}

\def\J{{\mathcal J}}

\def\R{{\mathbb R}}

\def\({{\Big(}}
\def\){{\Big)}}

\def\c1{{\C_c^1}}

\def\dz{{\rm d}z}
\def\dx{{\rm d}x}
\def\dy{{\rm d}y}

\def\cns{{C_{N,s}}}

\def\Hsrn{{H^s(\mathbb{R}^N)}}

\def\uss{{[u]_{s}^2}}

\def\xsom{{H^s_0(\Omega)}}
\def\into{{\displaystyle \int_{\Omega}}}
\def\uxys{{\frac{|u(x)-u(y)|^2}{|x-y|^{N+2s}}}}

\usepackage[foot]{amsaddr}
\usepackage[pagewise]{lineno}

\makeatother

\date{}
\begin{document}
\title[Isoperimetric inequalities]{Isoperimetric inequalities for the fractional composite membrane problem}

\author[Mrityunjoy Ghosh]{Mrityunjoy Ghosh \orcidlink{0000-0003-0415-2821}}

\address{Tata Institute of Fundamental Research,
Centre for Applicable
Mathematics\\
Post Bag No. 6503, Sharadanagar,
Bangalore 560065, India.}

\email{ghoshmrityunjoy22@gmail.com}




\begin{abstract}

In this article, we investigate  some isoperimetric-type inequalities related to the first eigenvalue of the fractional composite membrane problem. First, we establish an analogue of the renowned Faber-Krahn inequality for the fractional composite membrane problem. Next, we investigate an isoperimetric inequality for the first eigenvalue of the fractional composite membrane problem on the intersection of two domains-a problem that was first studied by Lieb \cite{Lieb1983} for the Laplacian. Similar results in the local case were previously obtained by Cupini-Vecchi \cite{Vecchi2019} for the composite membrane problem. Our findings provide further insights into the fractional setting, offering a new perspective on these classical inequalities.

\end{abstract}

\maketitle


\section{Introduction}
Let $\Om\subset \R^N$ $(N\geq 2)$ be a bounded, Lipschitz domain. For $s\in (0,1)$ and a measurable set $D\subset \Om,$ we consider the following eigenvalue problem for the fractional Laplacian with Dirichlet boundary conditions:
\begin{equation}\tag{$\mathcal{P}$}\label{Problem}
    \left\{\begin{aligned}
    (-\De)^s u +\al \chi_D u &=\la u \quad\; \text{in}\; \Om,\\
    u &=0\quad\;\;\; \text{on}\; \Om^c,
\end{aligned}\right.
\end{equation}
where $\al>0,\la\in \R$ and $\chi_D$ denotes the characteristic function of $D.$ It is well-known that there exists a least positive real number $\la_{\Om}(\al,D),$ which is the first eigenvalue of \eqref{Problem}, having the following characterization; cf. \cite[Section 3]{BonderPezzo}:
\begin{equation}\label{la1}
    \la_{\Om}(\al,D)=\inf_{u\in H^s_0(\Om)\setminus \{0\}}\left\{\frac{\frac{\cns}{2}\uss+\al \into \chi_{D}|u|^2\dx}{\into |u|^2\dx}\right\},
\end{equation}
where $\cns$ is a normalization constant, $H^s_0(\Om)$ is the fractional Sobolev space (see Section \ref{Frac_sp} for definitions) and $[u]_s$ is the Gagliardo semi-norm of $u$ given by
$$[u]_s:=\left(\intRn\intRn\uxys\dx\dy\right)^{\frac{1}{2}}.$$
Moreover, $\la_{\Om}(\al,D)$ is attained for some $u\in H^s_0(\Om).$

For a fixed $c\in[0,|\Om|],$ let us denote by $\Ds_{c,\Om}$ the following collection of subsets of $\Om$:
\begin{equation}\label{D_c}
    \Ds_{c,\Om}=\{D\subset\Om: D\;\text{is measurable and}\;|D|=c\},
\end{equation}
where $|\cdot|$ is the $N$-dimensional Lebesgue measure. Let $\La_{\Om}(\al,c)$ be defined as
\begin{equation}\label{Lam_om}
    \La_{\Om}(\al,c)=\inf_{D\in \Ds_{c,\Om}}\{\la_{\Om}(\al,D)\}.
\end{equation}
The optimization problem \eqref{Lam_om} is commonly known as the \textit{composite membrane problem}; cf. \cite{Chanillo2000}. We know that there exists $D\in \Ds_{c,\Om}$ such that $\La_{\Om}(\al,c)$ is achieved (see Proposition \ref{prop:property}) and we call it as an \textit{optimal domain}. For such optimal domain $D,$ if $u\in H^s_0(\Om)$ is an associated first eigenfunction of \eqref{Problem}, then we call $(u,D)$ an \textit{optimal pair} for the minimzation problem \eqref{Lam_om}. Observe that the optimal domains are not unique in general, which follows from the asymmetry in some domains; see, for instance, \cite[Theorem 1.4]{LeeLee}. However, the uniqueness might occur in some scenario; for example, if $\Om$ is a ball \cite[Corollary 1.3]{LeeLee}.

The composite membrane problem has extensive applications in diffusive logistic equations, population dynamics, etc. We refer to \cite{Chanillo2000,UNm, MG2023,Verzini2020, MazariDavoli, Lamboley, Yanagida, Kao2021, Colasuonno2019, Chanillo2008, Shahgholian} and the references therein for more applications along with recent developments about the composite membrane problem in the local case. However, the composite membrane problem in the non-local case has not been investigated much in the literature except a few; see, \cite{LeeLee, BonderPezzo, Vecchi2019, Bonder2021}.

The aim of the present article is to provide some isoperimetric-type inequalities for the fractional composite membrane problem. Firstly, we study Faber-Krahn type inequality for $\La_\Om(\al,c)$ in the same spirit as the classical one; cf. \cite{Faber, Krahn1926} or \cite[Chapter 3]{Henrot2006}. Secondly, we investigate the isoperimetric-type inequality related to $\La_\Om(\al,c)$ for the intersection of two domains. Such a result was first studied by Lieb \cite{Lieb1983} for the first eigenvalue of the Laplacian in the classical case.

\subsection{Faber-Krahn inequality} 
In \cite[Theorem 1.1]{Vecchi2019}, the authors considered the local version of \eqref{Problem} and they have established an analogue of the celebrated \textit{Faber-Krahn inequality} for the composite membrane problem. In this section, we derive a similar result for the non-local eigenvalue problem \eqref{Problem} for the fractional Laplacian.
\begin{theorem}\label{theo:Faber}
     Let $\Om\subset \R^N$ be a bounded, Lipschitz domain and $\Om^*$ be the ball centered at the origin satisfying $|\Om|=|\Om^*|$. Suppose $c\in (0,|\Om^*|)$ and $\al>0$ are given. Then, we have the following
    \begin{equation}\label{FaberKrahn_ineq}
        \La_{\Om^*}(\al,c)\leq \La_{\Om}(\al,c),
    \end{equation}
    where $\La_{\Om}(\al,c)$ is defined by \eqref{Lam_om}. Furthermore, equality occurs in \eqref{FaberKrahn_ineq} if and only if $\Om=\Om^*$ (up to translation).
\end{theorem}

 The approach to prove Theorem \ref{theo:Faber} is based on standard symmetrization arguments, which is also used in the local case. Although the method we follow is close to the local case (see \cite{Vecchi2019}), there will be some differences while deriving the equality case. The inequality part in Theorem \ref{theo:Faber} comes from the Hardy-Littlewood and fractional Polya-Sze\"{g}o inequality related to Schwarz symmetrization. In order to get the rigidity result in Theorem \ref{theo:Faber}, we use the characterizations of the equality case of fractional Polya-Sze\"{g}o inequality, thanks to \cite[Theorem A.1]{Frank2018}. The novelty of this approach is that our proof works for any $\al>0$, which is unlike the local case. Precisely, in the local case, the rigidity result holds provided the parameter $
 \al<\bar{\al}_{\Om}(c),$ where $\bar{\al}_{\Om}(c)$ is the unique number satisfying $\La_{\Om}(\bar{\al}_{\Om}(c),c)=\bar{\al}_{\Om}(c)$ (cf. \cite[Proposition 10]{Chanillo2000}). Such restriction to the range of the parameter $\al$ appears when one needs to use the fact that $\La_{\Om}(\al,c)>\al$ for $\al<\bar{\al}_{\Om}(c)$; see, for instance \cite[Remark 5]{Vecchi2019}. However, we do not employ this fact in our proof thereby leading to a more general approach that does not depend on the parameter $\al.$

\subsection{Lieb inequality} In this section, we are interested in Lieb type (see \cite{Lieb1983} for the Laplacian and \cite{Duong} for the fractional Laplacian) inequality for the fractional composite membrane problem. First, we introduce some notations. 

For $i=1,2,$ let $\Om_i\subset \R^N$ be two bounded, Lipschitz domain. Suppose $\al_i>0$ and $c_i\in (0,|\Om_i|),$ for $i=1,2,$ are given. By Proposition \ref{prop:property}, we know that $\La_{\Om_i}(\al_i,c_i)$ is achieved. For $x\in \R^N,$ let us define the translation of $\Om_2$ by $x$ as below:
\begin{equation}\label{translation}
    \Om_{2,x}:=\{y\in \R^N: y=z+x,\;\text{for some}\;z\in\Om_2\}.
\end{equation}

\begin{theorem}\label{theo:Lieb}
    Let $\Om_i,$ $\al_i$ and $c_i$ ($i=1,2$) be as given above. Also, let $D_i\in \Ds_{c_i,\Om_i}$ be an optimal domain for \eqref{Lam_om} on $\Om_i,$ where $\Ds_{c_i,\Om_i}$ is given by \eqref{D_c}. Then there exists a set $\A\subset \R^N$ of positive measure such that $|\Om_1\cap \Om_{2,x}|>0,$ for a.e. $x\in \A$ and 
    $$\La_{\Om_1\cap \Om_{2,x}}(\al,c)<\La_{\Om_1}(\al_1,c_1)+\La_{\Om_2}(\al_2,c_2),$$
    for each $\al\in (0,\al_1+\al_2)$ and $c\in (0,|D_1\cap D_{2,x}|),$ where $\Om_{2,x}$ is defined by \eqref{translation}.
\end{theorem}

In \cite[Theorem 1.2]{Vecchi2019}, the authors provided an analogue of the aforementioned Lieb inequality for the composite membrane problem in the local setting. We establish an analogous Lieb-type inequality in the fractional case for $\La_\Om(\al,c)$ related to \eqref{Problem}. The proof of the above theorem will be based on adapting the elementary approach that was introduced by Lieb \cite{Lieb1983}. The main idea is to produce a suitable test function by taking the product of the first eigenfunctions of $\Om_i$ ($i=1,2$) with translations. However, to do so, there will be some difficulties due to the non-local nature of the underlying operator; see, for instance, Section \ref{sec:Lieb} (Lemma \ref{Estimate}).

The remainder of the article is structured as follows. In the next section, we recall some preliminary facts about the fractional Sobolev spaces and symmetrizations. The proofs of Theorem \ref{theo:Faber} and Theorem \ref{theo:Lieb} are presented in Section \ref{sec:Main}.

\section{Preliminaries}
In this section, we review several auxiliary results that will serve as essential tools in the derivation of our main results.

\subsection{Fractional Sobolev spaces}\label{Frac_sp} 
For $s\in (0,1),$ let
\begin{equation*}
[u]_s :=\left(\intRn\intRn\uxys\dx\dy\right)^{\frac{1}{2}}
\end{equation*}
be the Gagliardo semi-norm of $u$ in $\R^N.$
The fractional Sobolev space $\Hsrn$ is defined as follows:
\begin{equation}\label{Sobolev_sp}
\begin{aligned}
    \Hsrn =\left\{u \in L^2(\R^N): [u]_s^2 <\infty\right\}.
    \end{aligned}
\end{equation}
For $u,v\in \Hsrn,$ let $\langle u,v\rangle_{\Hsrn}$ be given by 
\begin{equation*}\label{inner_pr}
    \langle u,v\rangle_{\Hsrn}:= \intRn u(x)v(x)\dx + \intRn\intRn \frac{(u(x)-u(y))(v(x)-v(y))}{|x-y|^{N+2s}}\dx\dy.
\end{equation*}
One can easily verify that $\langle \cdot,\cdot\rangle_{\Hsrn}$ defines an inner product on $\Hsrn$. Let $||\cdot||_{\Hsrn}$ be the norm induced by the aforementioned inner product, i.e., for $u\in \Hsrn,$ we have
\begin{equation*}
||u||_{\Hsrn}=\left(||u||^2_{L^2(\R^N)}+[u]_s^2\right)^\frac{1}{2}.
\end{equation*}
Then $\Hsrn$ is a Hilbert space with respect to the norm $||\cdot||_{\Hsrn}$; see, for instance, \cite[Section 1.2.2]{Radulescu} or \cite[Lemma 6 and Lemma 7]{Servadei2012}. Let $\Om\subset \R^N$ be a bounded domain. We define the Sobolev space $\xsom$, which is a subspace of $\Hsrn,$ as below
\begin{equation*}
    \xsom =\{u \in \Hsrn: u\equiv 0 \;\text{in }\;\Om^c\}.
\end{equation*}
 Observe that, for $u\in \xsom,$ we have 
\begin{equation*}\label{hs_norm}
||u||_{\xsom}:=||u||_{\Hsrn}=\left(||u||^2_{L^2(\Om)}+[u]_s^2\right)^\frac{1}{2}.
\end{equation*}
Let us recall some known properties of the 
Sobolev space $\xsom$ defined above. The space $\xsom$ is a Hilbert space with respect to the norm $||\cdot||_{\xsom};$ see, for instance, \cite[Lemma 1.28 and Lemma 1.29]{Radulescu}. Furthermore, the embedding $\xsom\hookrightarrow L^r(\Om)$ is compact for all $r\in[1,2_s^*),$ where $2_s^*:=\frac{2N}{N-2s}$ is the fractional critical Sobolev exponent of $2,$ and the embedding $\xsom\hookrightarrow L^{2_s^*}(\Om)$ is continuous; cf. \cite[Lemma 9]{Servadei2015} or \cite[Lemma 1.31]{Radulescu}.

\subsection{Symmetrization} We start this section with some elementary definitions and notations regarding symmetrization; we refer to \cite{Kesavan,LossLieb,Baernstein} for more details.
For a given function $f:\Om\ra \R,$ we use the following notations:
$$\left\{f>t\right\}:=\{x\in \Om:f(x)>t\}\;\text{and}\;\left\{f<t\right\}:=\{x\in \Om:f(x)<t\}.$$
Throughout this section, we assume 
$|\left\{f>t\right\}|<+\infty,$ for each $t>\text{ess inf}(f)$.

\begin{definition}[Rearrangement]
Let $\Om$ be a measurable set and $f:\Om\ra \R$ be a measurable function.
\begin{enumerate}[(i)]
    \item The \textit{one-dimensional decreasing rearrangement} $f^\#:[0,|\Om|]\ra [-\infty,+\infty]$ of $f$ is defined as below:
    \begin{align*}
 f^\#(\xi)= \begin{cases} 
      \text{ess sup}(f), & \text{if}\;\xi=0, \\
      \inf\{t:|\left\{f>t\right\}|<\xi\}, & \text{if}\;0<\xi<|\Om|,\\
      \text{ess inf}(f), & \text{if}\;\xi=|\Om|.
   \end{cases}
\end{align*}

\item The \textit{one-dimensional increasing rearrangement} $f_\#:[0,|\Om|]\ra [-\infty,+\infty]$ of $f$ is defined as follows: 
\begin{align*}
 f_\#(\xi)= \begin{cases} 
      \text{ess inf}(f), & \text{if}\;\xi=0, \\
      \inf\{t:|\left\{f<t\right\}|>\xi\}, & \text{if}\;0<\xi<|\Om|,\\
      \text{ess sup}(f), & \text{if}\;\xi=|\Om|.
   \end{cases}
\end{align*}
\end{enumerate}
\end{definition}

\begin{definition}[Schwarz symmetrization]
    Let $\Om$ be a measurable set. Also, let $\Om^*$ be the ball centered at the origin such that $|\Om^*|=|\Om|$ if $|\Om|<+\infty$ and $\Om^*=\R^N$ if $|\Om|=+\infty.$ The \textit{decreasing Schwarz symmetrization} $f^*$ and \textit{increasing Schwarz symmetrization} $f_*$ are defined as 
    $$f^*(x)=f^\#(\om_N|x|^N),\;f_*(x)=f_\#(\om_N|x|^N),\;\text{for}\;x\in \Om^*,$$
    where $\om_N$ is the volume of the unit ball in $\R^N.$
\end{definition}

Now, we state the well-known Hardy-Littlewood and Polya-Sze\"{g}o inequality involving the Schwarz symmetrization.

\begin{lemma}[Hardy-Littlewood inequality]\label{Hardy_Littlewood}
    Let $f,g\in L^2(\Om)$. If $f_*$ denotes the increasing Schwarz symmetrization of $f$ and $g^*$ denotes the decreasing Schwarz symmetrization of $g,$ then 
    \begin{equation}
        \int_{\Om^*} f_*g^*\dx\leq \int_{\Om} fg\dx.
    \end{equation}
\end{lemma}
\begin{proof}
    The proof is a consequence of \cite[Corollary 1.4.1 and eq. (1.3.3)]{Kesavan} or \cite[Corollary 2.18]{Baernstein}.
\end{proof}

\begin{lemma}[Polya-Sze\"{g}o inequality]\label{polya}
    Let $f\in \Hsrn$. Then we have 
    $$[f^*]_s^2\leq [f]_s^2.$$
    Moreover, equality occurs in the above inequality if and only if f is a translate of a symmetric decreasing function.
\end{lemma}
\begin{proof}
    The first assertion follows from \cite[Theorem 9.2]{Almgren}. For the proof of the equality case, we refer the reader to \cite[Theorem A.1]{FrankSeiringer}.
\end{proof}

\subsection{Properties of the minimizers} Here, we state some properties of the optimal domains and the corresponding first eigenfunctions; cf. \cite{LeeLee}.
\begin{proposition}\label{prop:property}
    Let $c\in [0,|\Om|]$ and $\al>0$. Then the following holds true:
    \begin{enumerate}[(i)]
        \item There exists an optimal pair $(u,D)$ in \eqref{Lam_om} such that $u$ has a constant sign in $\Om.$ 

        \item The function $(\al,c)\mapsto \La_{\Om}(\al,c)$ is increasing in $(0,\infty)\times (0,|\Om|)$ with respect to both the variable $\al$ and $c.$
        
    \end{enumerate}
\end{proposition}

\begin{proof}
     The proof of $(i)$ can be found in \cite[Theorem 1.1]{LeeLee} and $(ii)$ is proved in \cite[Lemma 4.3]{LeeLee}.
\end{proof}

\section{Proofs of the Main Results}\label{sec:Main}

\subsection{Proof of Faber-Krahn inequality} 
In this section, we give a proof of Theorem \ref{theo:Faber}. In the proof, we use standard symmetrization arguments along with the help of the structure of an optimal pair associated to \eqref{Lam_om}.

\begin{proof}[Proof of Theorem \ref{theo:Faber}]
   By hypothesis, we have $c\in (0,|\Om^*|),$ where $\Om^*$ is the ball centered at the origin such that $|\Om|=|\Om^*|,$ and $\al>0.$ Let $D\in \Ds_{c,\Om}$ be an optimal domain and $(u,D)$ be an optimal pair in \eqref{Lam_om}. Then, using the characterization \eqref{la1}, we have 
   \begin{equation}\label{Equal_om}
       \La_{\Om}(\al,c)=\la_{\Om}(\al,D)=\frac{\frac{\cns}{2}\uss+\al \into \chi_{D}|u|^2\dx}{\into |u|^2\dx}.
   \end{equation}
   Let $D_*\subset \Om^*$ be defined as below
   \begin{equation}
       \chi_{D_*}(x)=(\chi_D)_*(x),\;\text{for}\;x\in \Om^*.
   \end{equation}
Now, it is easy to observe that $|D_*|=|D|$ and hence, $D_*\in \Ds_{c,\Om^*}.$ Therefore, using Hardy-Littlewood (Lemma \ref{Hardy_Littlewood}) and Polya-Sze\"{g}o (Lemma \ref{polya}) inequality together in \eqref{Equal_om}, we derive
   \begin{align}\label{Var_ineq}
       \La_{\Om}(\al,c)=\frac{\frac{\cns}{2}\uss+\al \into \chi_{D}|u|^2\dx}{\into |u|^2\dx} & \geq \frac{\frac{\cns}{2}[u^*]_s^2+\al \displaystyle \int_{\Om^*} (\chi_{D})_*|u^*|^2\dx}{\displaystyle\int_{\Om^*} |u^*|^2\dx}\nonumber \\
       & \geq \la_{\Om^*}(\al,D_*)\geq \La_{\Om^*}(\al,c),
   \end{align}
   where the last two inequalities follow from the definitions of $\la_{\Om^*}(\al,D_*)$ and $\La_{\Om^*}(\al,c),$ respectively. Hence, it follows that $$\La_{\Om^*}(\al,c)\leq \La_{\Om}(\al,c).$$
   Next, we prove the equality case. If $\Om=\Om^*,$ then obviously we have the equality. Conversely, assume that $\La_{\Om}(\al,c)=\La_{\Om^*}(\al,c)$. Therefore, in view of  \eqref{Var_ineq}, we must have 
   \begin{align}
       \uss=[u^*]_s^2.
   \end{align}
   Thus, using the characterization of the equality case of Polya-Sze\"{g}o inequality (Lemma \ref{polya}), we conclude that $u=u^*$ (up to a translation) and consequently, $\Om=\Om^*$ (up to a translation). This completes the proof.
\end{proof}

\begin{remark}
    As highlighted in the introduction, to derive the equality case in the above Faber-Krahn inequality, we use the characterization of the equality case in fractional Polya-Sze\"{g}o inequality; cf. \cite[Theorem A.1]{FrankSeiringer}. In particular, this provides a direct and easier proof as compared to the local case; cf. \cite[Theorem 1.1]{Vecchi2019}. 
\end{remark}

\subsection{Proof of Lieb inequality}\label{sec:Lieb} 
In this section, we establish Theorem \ref{theo:Lieb}. First, we state a technical lemma that is inspired by the proof of \cite[Theorem 1.2]{Duong}. However, we derive an improved version that actually gives equality. Moreover, the equality indeed has its own importance as it provides a non-local analogue of the identity (cf. \cite[eq. (1.12)]{Lieb1983}) used by Lieb \cite{Lieb1983} for the Laplacian that follows from the product rule. The proof will be based on similar arguments as used in \cite[Theorem 1.2]{Duong} but with some additional observations that lead to equality. We supply a proof here to distinguish the equality case.

\begin{lemma}\label{Estimate}
    Let $u_i\in H^s(\R^N),$ for $i=1,2,$ be such that they are non-negative a.e. For $x\in \R^N$, define $u_x:\R^N\ra \R$ as follows:
\begin{equation}
    u_x(y)=u_1(y)u_2(y-x),\;\text{where}\;y\in \R^N.
\end{equation} 
Then we have the following:
$$\intRn [u_x]_s^2\dx = [u_1]_s^2||u_2||^2_{L^2(\R^N)}+[u_2]_s^2||u_1||^2_{L^2(\R^N)}.$$
In particular, if $||u_i||_{L^2(\R^N)}=1,$ for $i=1,2,$ it holds
$$\intRn [u_x]_s^2\dx = [u_1]_s^2+[u_2]_s^2.$$
\end{lemma}
\begin{proof}
    Let $x\in \R^N$ be fixed. From the definition of $u_x,$ we get
\begin{align}
    [u_x]_s^2 &=\intRn\intRn \frac{|u_1(y)u_2(y-x)-u_1(z)u_2(z-x)|^2}{|y-z|^{N+2s}}\dy\dz\\
    &= \intRn\intRn \frac{\left\{(u_1(y)-u_1(z))u_2(y-x)-u_1(z)(u_2(z-x)-u_2(y-x))\right\}^2}{|y-z|^{N+2s}}\dy\dz\\
    &= \underbrace{\intRn\intRn \frac{(u_1(y)-u_1(z))^2 (u_2(y-x))^2}{|y-z|^{N+2s}}\dy\dz}_{:=\J_1(x)} \\
    & \hspace{1cm} + \underbrace{2 \intRn\intRn \frac{\left\{u_1(z)(u_1(y)-u_1(z))u_2(y-x)(u_2(z-x)-u_2(y-x)) \right\}^2}{|y-z|^{N+2s}}\dy\dz}_{:=\J_2(x)}\\
    & \hspace{2cm} + \underbrace{\intRn\intRn \frac{(u_1(z)(u_2(z-x)-u_2(y-x))^2}{|y-z|^{N+2s}}\dy\dz}_{:=\J_3(x)}.\label{u_x}
\end{align}
Let us analyse each of $\J_1(x),\J_2(x)$ and $\J_3(x)$ separately. We have 
\begin{align}\label{J_1}
    \intRn \J_1(x)\dx=\intRn\intRn\intRn \frac{(u_1(y)-u_1(z))^2 (u_2(y-x))^2}{|y-z|^{N+2s}}\dy\dz\dx=[u_1]_s^2 ||u_2||^2_{L^2(\R^N)},
\end{align}
where we use Fubini's theorem to obtain the last equality. Next, by similar arguments after making a change of variable, one can easily obtain
\begin{align}\label{J_3}
    \intRn\J_3(x)\dx=[u_2]_s^2||u_1||^2_{L^2(\R^N)}.
\end{align}
Now, we claim that 
\begin{equation}\label{J_2}
    \intRn\J_2(x)\dx= 0.
\end{equation}
Indeed, again applying Fubini's theorem, we get
\begin{align}
    \intRn\J_2(x)\dx &=2\intRn \intRn\intRn \frac{\left\{u_1(z)(u_1(y)-u_1(z))u_2(y-x)(u_2(z-x)-u_2(y-x)) \right\}^2}{|y-z|^{N+2s}}\dy\dz\dx\\
    & =2\intRn\intRn \frac{u_1(z)(u_1(y)-u_1(z))}{|y-z|^{N+2s}} \left\{\underbrace{\intRn u_2(y-x)(u_2(z-x)-u_2(y-x))\dx}_{:=\J} \right\}\dy\dz.
\end{align}
Now,
\begin{align}
    \J &=\intRn u_2(y-x)(u_2(z-x)-u_2(y-x))\dx\\
    &=\intRn u_2(z-x)u_2(y-x)\dx- \intRn (u_2(y-x))^2\dx = ||u_2||_{L^2(\R^N)}^2-||u_2||_{L^2(\R^N)}^2=0,
\end{align}
where we use the fact $\intRn u_2(z-x)u_2(y-x)\dx=||u_2||_{L^2(\R^N)}^2$ that comes from the equality case in the H\"{o}lder's inequality. Therefore, the claim \eqref{J_2} follows immediately. Using \eqref{J_1}, \eqref{J_3} and \eqref{J_2} altogether in \eqref{u_x}, we obtain
$$\intRn [u_x]_s^2\dx = [u_1]_s^2||u_2||^2_{L^2(\R^N)}+[u_2]_s^2||u_1||^2_{L^2(\R^N)},$$
as claimed in the statement. The proof is finished.
\end{proof}

Next, we prove Theorem \ref{theo:Lieb}.
\begin{proof}[Proof of Theorem \ref{theo:Lieb}] By assumption, $\La_{\Om_i}(\al_i,c_i)$ ($i=1,2$) is achieved for $D_i$. Let $u_i\in H^s_0(\Om_i),$ for $i=1,2,$ be a positive, $L^2$-normalized eigenfunction of \eqref{Problem} corresponding to $\La_{\Om_i}(\al_i,c_i),$ i.e., $(u_i,D_i)$ is an optimal pair for \eqref{Lam_om} on $\Om_i$ such that
\begin{equation}\label{ui_norm}
    ||u_i||_{L^2(\Om_i)}=1.
\end{equation} 
Thus we have
\begin{equation}\label{Var_i}
    \La_{\Om_i}(\al_i,c_i)=\frac{\cns}{2}[u_i]_s^2+\al_i \int_{\Om_i} \chi_{D_i}|u_i|^2\dx,\;\text{for}\;i=1,2.
\end{equation}
For $x\in \R^N$, we define $u_x:\R^N\ra \R$ as follows:
\begin{equation}
    u_x(y)=u_1(y)u_2(y-x),\;y\in \R^N.
\end{equation}
We divide the proof in several steps.\\
\underline{\textit{Step-I}}: Our aim is to show that $u_x\in H^s_0(\Om_1\cap\Om_{2,x}),$ for almost all $x\in \R^N.$ Observe that $u_x$ is well defined. Let 
\begin{equation}\label{Ux}
U(x):=\int_{\R^N}|u_x(y)|^2\dy,\;\text{for}\;x\in\R^N.
\end{equation}
Then using Fubini's theorem and \eqref{ui_norm}, we get 
\begin{align}
||U||_{L^1(\R^N)} &=\int_{\R^N}\left(\int_{\R^N}|u_x(y)|^2\dy\right)\dx\\
&=\int_{\R^N} |u_1(y)|^2 \left(\int_{\R^N} |u_2(y-x)|^2\dx\right)\dy=1.\label{U_x_1}
\end{align}
Therefore, we must have $U(x)<\infty,$ for almost all $x\in\R^N$ and consequently, one has $u_x\in L^2(\R^N),$ for almost all $x\in\R^N.$ Recall that, for $x\in \R^N$, 
\begin{equation}
     [u_x]_s=\left(\intRn\intRn \frac{|u_x(y)-u_x(z)|^2}{|y-z|^{N+2s}}\dy\dz\right)^{\frac{1}{2}}.
\end{equation}
Observe that the function $x\mapsto [u_x]_s$ belongs to $L^2(\R^N);$ thanks to Lemma \ref{Estimate}. Indeed, applying Lemma \ref{Estimate} along with \eqref{ui_norm}, one has
\begin{align}\label{ux_1}
    \int_{\R^N}[u_x]_s^2\dx= [u_1]_s^2+[u_2]_s^2<\infty.
\end{align}
Moreover, since $u_i=0$ on $\Om_i^c$ ($i=1,2$), it gives $u_x(y)=0,$ for all $y\in (\Om_x)^c,$ where
$$\Om_x:=\Om_1\cap \Om_{2,x}.$$
As a result, we have $u_x\in H^s_0(\Om_x),$ for almost all $x\in \R^N.$ \\
\underline{\textit{Step-II}}: Since $D_1\cap D_{2,x}\subset \Om_x,$ Fubini's theorem together with \eqref{ui_norm} imply  
\begin{align}
    \intRn &\left(\int_{\Om_x}\chi_{D_1\cap D_{2,x}(y)}|u_x(y)|^2\dy\right) \dx \\
    & =\intRn\intRn \chi_{D_1}(y)\chi_{D_2}(y-x)(u_1(y))^2 (u_2(y-x))^2\dy\dx\\
    & = \int_{\R^N} \chi_{D_1}(y) (u_1(y))^2 \left(\int_{\R^N}\chi_{D_2}(y-x)(u_2(y-x))^2\dx \right)\dy \\
    & = \left(\int_{D_1} (u_1(y))^2\dy \right)\left(\int_{D_2}(u_2(z))^2\dz \right)\\
    &= \left[\frac{\al_1}{\al_1+\al_2}+\frac{\al_2}{\al_1+\al_2}\right]\left(\int_{D_1} (u_1(y))^2\dy \right)\left(\int_{D_2}(u_2(z))^2\dz \right)\\
    &\leq \frac{\al_1}{\al_1+\al_2}\left(\int_{D_1} (u_1(y))^2\dy \right)\left(\int_{\Om_2}(u_2(z))^2\dz \right) \\
    & \hspace{4cm} + \frac{\al_2}{\al_1+\al_2}\left(\int_{\Om_1} (u_1(y))^2\dy \right)\left(\int_{D_2}(u_2(z))^2\dz \right)\\
    &=\frac{\al_1}{\al_1+\al_2}\int_{\Om_1} \chi_{D_1}(y)(u_1(y))^2\dy+ \frac{\al_2}{\al_1+\al_2} \int_{\Om_2}\chi_{D_2}(y)(u_2(z))^2\dz.\label{Last_term}
\end{align}
Let 
\begin{equation}
    W(x):=\frac{\cns}{2} [u_x]_s^2+ (\al_1+\al_2)\int_{\Om_x}\chi_{D_1\cap D_{2,x}}(y)|u_x(y)|^2\dy,\;\text{for}\;x\in\R^N.
\end{equation}
Then using \eqref{ux_1}, \eqref{Last_term} and the characterizations \eqref{Var_i} of $\La_{\Om_i}(\al_i,c_i)$ $(i=1,2)$, we obtain
\begin{align}
    \intRn W(x)\dx & \leq \frac{\cns}{2}[u_1]_s^2 +\al_1 \int_{\Om_1} \chi_{D_1}(y)(u_1(y))^2\dy  \\
    & \hspace{2cm} + \frac{\cns}{2} [u_2]_s^2 + \al_2 \int_{\Om_2}\chi_{D_2}(y)(u_2(z))^2\dz\\
    & = \La_{\Om_1}(\al_1,c_1)+\La_{\Om_2}(\al_2,c_2):=\La\;\text{(say)}.
\end{align}
Therefore, in view of \eqref{U_x_1}, it holds 
\begin{equation}\label{WU}
    \intRn (W(x)-\La U(x))\dx\leq 0.
\end{equation}
\underline{\textit{Step-III}}: We show that
$W(x)\neq \La U(x),\;\text{for almost all}\;x\in \R^N.$
On contrary, let us assume that $W(x)=\La U(x),$ for almost all $x\in\R^N.$ First, we claim that $U>0$ on some open set. Let $\widetilde{\chi}(x):=(\chi_{\Om_1}\ast \chi_{\Om_2})(x)=|\Om_x|.$ Then $\widetilde{\chi}\in \C_c(\R^N)$ (continuous function with compact support) and hence for given $\ep>0,$ there exists an open set $A_\ep$ such that $0<\widetilde{\chi}(x)<\ep,$ for all $x\in A_\ep,$ i.e.,
\begin{align}
    0<|\Om_x|<\ep,\;\text{for all}\;x\in A_\ep.\label{Om_ep}
\end{align} 
Moreover, for each fixed $x\in A_\ep,$ we have $u_1(y)>0$ and $u_2(y-x)>0,$ for all $y\in \Om_x.$ Therefore, we must have $U>0$ on $A_\ep.$ Suppose $c_x:=|D_1\cap D_{2,x}|.$ Recall that  $u_x\in H^s_0(\Om_x),$ for almost all $x\in \R^N.$ Hence, making use of the variational characterizations \eqref{Lam_om} and \eqref{la1} of $\La_{\Om_x}(\al_1+\al_2,c_x)$ and $\la_{\Om_x}(\al_1+\al_2,D_1\cap D_{2,x})$, respectively, we get
\begin{align}
    \La_{\Om_x}(\al_1+\al_2,c_x) & \leq \la_{\Om_x}(\al_1+\al_2,D_1\cap D_{2,x})\\
    & \leq \frac{\frac{\cns}{2} [u_x]_s^2+ (\al_1+\al_2)\int_{\Om_x}\chi_{D_1\cap D_{2,x}(y)}|u_x(y)|^2\dy}{\int_{\Om_x}|u_x(y)|^2\dy}\\
    & =\frac{W(x)}{U(x)}=\La,\;\text{for almost all}\;x\in A_\ep.\label{LaLa}
\end{align}
Let $\la_1(\Om_x)$ be the first Dirichlet eigenvalue of the fractional Laplacian on $\Om_x,$ i.e., $\la_1(\Om_x)$ is the first eigenvalue of 
\begin{align}
    (-\De)^s u &=\la u \quad\; \text{in}\; \Om_x,\\
    u &=0\quad\;\;\; \text{on}\; \Om_x^c.
\end{align}
Then, with the help of the variational characterization of the first eigenvalue, it implies: for each $D\in \Ds_{c_x,\Om_x},$
\begin{align}
    \la_1(\Om_x) &=\inf_{u\in H^s_0(\Om_x)\setminus \{0\}}\left\{\frac{\frac{\cns}{2}\uss}{\int_{\Om_x} |u|^2\dx}\right\} \\
    &\leq \inf_{u\in H^s_0(\Om_x)\setminus \{0\}}\left\{\frac{\frac{\cns}{2}\uss+(\al_1+\al_2) \int_{\Om_x} \chi_{D}|u|^2\dx}{\int_{\Om_x} |u|^2\dx}\right\}.
\end{align}
Consequently, we get $\la_1(\Om_x)\leq \La_{\Om_x}(\al_1+\al_2,c_x)$ and hence, using classical fractional Faber-Krahn inequality \cite[Section 5.2.1 or Theorem 5.2.4]{Frank2018} we arrive at
\begin{equation}
    \frac{C(\Om_x, s)}{|\Om_x|^{\frac{2s}{N}}} \leq \la_1(\Om_x)\leq \La_{\Om_x}(\al_1+\al_2,c_x),\label{Frac_Faber}
\end{equation}
where $C(\Om_x, s)$ is a constant that depends only on $\Om_x$ and $s.$ Now, using both \eqref{Om_ep} and \eqref{LaLa} in \eqref{Frac_Faber} yields
\begin{align}
    \frac{C(\Om_x, s)}{\ep^{\frac{2s}{N}}} \leq \La,
\end{align}
which is a contradiction as $\ep>0$ is arbitrary. Therefore, we must have $W(x)\neq \La U(x),$ for almost all $x\in \R^N.$ Thus in view of \eqref{WU}, there exists a set $\A\subset \R^N$ of positive measure so that 
\begin{equation}
    W(x)< \La U(x),\;\text{for almost all}\;x\in \A.
\end{equation}
Hence, from \eqref{LaLa}, one has
\begin{equation}
    \La_{\Om_x}(\al_1+\al_2,c_x)\leq \frac{W(x)}{U(x)}<\La,\;\text{for almost all}\;x\in \A.
\end{equation}
Finally, Proposition \ref{prop:property}-$(ii)$ implies $\La_{\Om_x}(\al,c)\leq \La_{\Om_x}(\al_1+\al_2,c_x),$ for all $\al\in (0,\al_1+\al_2)$ and $c\in (0,c_x).$ Therefore, recalling the definition of $\Om_x$, $c_x$ and $\La,$ we get
$$\La_{\Om_1\cap \Om_{2,x}}(\al,c)< \La_{\Om_1}(\al_1,c_1)+\La_{\Om_2}(\al_2,c_2),\;\text{for all}\;\al\in (0,\al_1+\al_2)\;\text{and}\;c\in (0,|D_1\cap D_{2,x}|).$$
This concludes the proof.

\end{proof}

\begin{remark}(Extension to mixed local-non local operator) We provide some comments regarding the extensions of Faber-Krahn and Lieb inequalities presented in Theorem \ref{theo:Faber} and \ref{theo:Lieb}, respectively, to more general operators. In \cite[Theorem 1.1]{ValdinociVecchi}, the authors established the Faber-Krahn inequality for the first eigenvalue of the mixed local-non local operator $-\De+(-\De)^s$. For such an operator, one can establish an analogue of Theorem \ref{theo:Faber} and \ref{theo:Lieb} by following the similar set of arguments as used in our proofs.
    
\end{remark}

\section{Acknowledgements}

The author is thankful to TIFR Centre for Applicable Mathematics (TIFR-CAM) for the financial support and for providing various academic resources.

\bibliographystyle{abbrvurl}
\bibliography{Reference}
\end{document}